\documentclass[11pt]{article}

\usepackage{amsmath, amssymb, amsthm} 

\usepackage{dsfont}
\usepackage{appendix}
\usepackage[none]{hyphenat}

\allowdisplaybreaks
\expandafter\let\expandafter\oldproof\csname\string\proof\endcsname
\let\oldendproof\endproof
\renewenvironment{proof}[1][\proofname]{%
	\oldproof[\bf #1]%
}{\oldendproof}

\allowdisplaybreaks

\theoremstyle{plain}
\newtheorem{theorem}{Theorem}[section]
\newtheorem{lemma}[theorem]{Lemma}
\newtheorem{claim}[theorem]{Claim}

\newcommand{\poly}{\text{poly}}
\RequirePackage[normalem]{ulem} 
\RequirePackage{color}\definecolor{RED}{rgb}{1,0,0}\definecolor{BLUE}{rgb}{0,0,1} 

\author{
	Lior Gishboliner\thanks{ETH Zurich, \emph{e-mail}: \textbf{\{lior.gishboliner,istvan.tomon\}@math.ethz.ch}}
	\and
	Istv\'an Tomon\footnotemark[1]
}

\title{On $3$-graphs with no four vertices spanning exactly two edges}

\begin{document}
	\maketitle
	\sloppy
	\begin{abstract}
		Let $D_2$ denote the $3$-uniform hypergraph with $4$ vertices and $2$ edges.
		Answering a question of Alon and Shapira, we prove an induced removal lemma for $D_2$ having polynomial bounds. We also prove an Erd\H{o}s-Hajnal-type result: every induced $D_2$-free hypergraph on $n$ vertices contains a clique or an independent set of size $n^{c}$ for some absolute constant $c > 0$. In the case of both problems, $D_2$ is the only nontrivial $k$-uniform hypergraph with $k\geq 3$ which admits a polynomial bound.
	\end{abstract}
	\section{Introduction}
	The famous triangle removal lemma of Ruzsa and Szemer\'edi \cite{RuzsaSz} started a new chapter in combinatorics. It states that if an $n$-vertex graph $G$ contains at most $\delta(\varepsilon) n^3$ triangles, then $G$ can be made triangle-free by deleting at most $\varepsilon n^2$ edges. A similar statement holds 
	when the triangle is replaced with an arbitrary graph $F$. 
	Alon, Fischer, Krivelevich and Szegedy \cite{AFKS} proved an analogous result for the much more challenging setting of induced subgraphs. This result, known as the induced removal lemma, states that if $G$ contains at most $\delta_F(\varepsilon) n^{v(F)}$ induced copies of a graph $F$, then $G$ can be made induced $F$-free by adding/deleting at most $\varepsilon n^2$ edges. A generalization to arbitrary hereditary graph properties was later obtained by Alon and Shapira \cite{AS_hereditary}. For more on graph removal lemmas, we refer the reader to \cite{CF}. 
	
	One of the major developments in extremal combinatorics in the last twenty years was the establishment of a hypergraph version of Szemer\'edi's regularity lemma, which made it possible to extend the results mentioned in the previous paragraph to $k$-uniform hypergraphs.
	A hypergraph analogue of the graph removal lemma was proved by Gowers \cite{Gowers06,Gowers07} and independently by Nagle, R\"odl, Schacht and Skokan \cite{NRS,Rodl_Skokan}. An analogue of the induced removal lemma, and more generally the Alon-Shapira theorem, was then obtained by Avart, R\"odl and Schacht \cite{ARS} for $3$-uniform hypergraphs and by R\"odl and Schacht \cite{Rodl_Schacht} in the general case. As an example, for a $k$-uniform hypergraph $F$, the induced $F$-removal lemma states that if a $k$-uniform hypergraph $H$ contains at most $\delta_F(\varepsilon) n^{v(F)}$ induced copies of $F$, then $H$ can be made induced $F$-free by adding/deleting at most $\varepsilon n^k$ edges.  
	
	The proofs of all of the above results rely on the regularity lemma of Szemer\'edi \cite{Szemeredi} or generalizations thereof. Consequently, the bounds on $\delta(\varepsilon)$ supplied by these proofs are quite poor. Even in the case of the triangle removal lemma, the best known bound, due to Fox \cite{Fox}, is $1/\delta \leq \text{tower}(O(\log 1/\varepsilon))$, where $\text{tower}(x)$ is a tower of $x$ exponents. Still, in some cases better bounds are known. This raised the natural question of characterizing the cases where the removal lemma has polynomial bounds, namely, when $1/\delta$ can be taken as a polynomial function of $1/\varepsilon$. For example, in the case of graphs, Alon \cite{Alon} showed that the $F$-removal lemma has polynomial bounds if and only if $F$ is bipartite. Results of Alon and Shapira \cite{AS_induced} and Alon and Fox \cite{AF} gave a nearly complete characterization for the induced-$F$-removal lemma. See also \cite{GS_C4,GS_poly} for related results.
	
	The question of polynomial removal lemmas for hypergraphs was studied by Alon and Shapira in \cite{AS_hypergraphs}. They showed that for $k \geq 3$ and for a $k$-uniform hypergraph $F$, the induced-$F$-removal lemma {\em does not} have polynomial bounds, unless $|V(F)| = k$ (the trivial case), or $k = 3$ and $F$ is the $3$-uniform hypergraph with $4$ vertices and $2$ edges. We will denote this hypergraph by $D_2$. Consequently, Alon and Shapira \cite{AS_hypergraphs} asked whether the induced-$D_2$-removal lemma has polynomial bounds. Our first result answers this question in the affirmative:

	\begin{theorem}\label{thm:removal}
		For every $\varepsilon > 0$ there are $\delta = \delta(\varepsilon) = \poly(\varepsilon) > 0$ and $n_0 = n_0(\varepsilon) = \poly(1/\varepsilon)$ such that the following holds. Let $H$ be a $3$-uniform hypergraph on $n \geq n_0$ vertices which contains less than $\delta n^4$ induced copies of $D_2$. Then $H$ can be made induced $D_2$-free by adding/deleting at most $\varepsilon n^3$ edges. 
	\end{theorem}
	
	Our second result is an Erd\H{o}s-Hajnal-type theorem for induced $D_2$-free graphs. The famous Erd\H{o}s-Hajnal conjecture, raised in \cite{EH}, states that for every graph $F$, every induced $F$-free graph $G$ on $n$ vertices should contain a {\em homogeneous set}, i.e. a clique or independent set, of size at least $n^c$, where $c = c(F) > 0$. If true, this would show that in a strong sense, graphs belonging to hereditary properties have much larger homogeneous sets than general graphs, which are only guaranteed to have homogeneous sets of size $\Theta(\log n)$. See the survey of Chudnovsky \cite{C14} as a (somewhat outdated) general reference.

	The analogue of the Erd\H{o}s-Hajnal problem for $3$-uniform hypergraphs was studied by Conlon, Fox and Sudakov \cite{CFS}. The situation here is somewhat different to graphs, since it is not even known what size of homogeneous sets is guaranteed to exist in every $n$-vertex $3$-uniform hypergraph. By the well known theorem of Erd\H{o}s and Rado \cite{ER}, every $3$-uniform hypergraph contains a homogeneous set of size $\Omega(\log \log n)$, and there exist $3$-uniform hypergraphs with no homogeneous set of size $C\sqrt{\log n}$. Finding the correct order of magnitude is a major open problem in Ramsey theory. 
	
	Given the above, a reasonable candidate for an Erd\H{o}s-Hajnal-type conjecture for $3$-uniform hypergraphs would be that every induced $F$-free $n$-vertex hypergraph contains a homogeneous set of size $\omega(\log \log n)$, however, this is believed to be false in general. Some results in this vein were obtained in \cite{AST,CFS}. A related natural question is for which hypergraphs $F$, an Erd\H{o}s-Hajnal result holds with {\em polynomial bounds}. 
	In the following theorem, we show that this is the case for $F = D_2$, namely that induced $D_2$-free hypergraphs have homogeneous sets of polynomial size.  
	
	\begin{theorem}\label{thm:EH}
		There exists a constant $c>0$ such that the following holds. Every induced $D_2$-free $3$-uniform hypergraph on $n$ vertices contains a clique or an independent set of size at least $n^{c}$. 
	\end{theorem}

	Theorem \ref{thm:EH} in fact allows us to characterize the $k$-uniform hypergraphs $F$ for $k\geq 3$ such that induced $F$-free $k$-uniform hypergraphs have homogeneous sets of polynomial size. These are exactly the hypergraphs $F$ with $|V(F)| = k$ (the trivial cases) and $F = D_2$.
	
	\begin{claim}
	    Let $k\geq 3$ and let $F$ be a $k$-uniform hypergraph. If $|V(F)|\geq k+1$ and $F\neq D_2$, then for every $n$ there exists a $k$-uniform hypergraph $H$ on $n$ vertices such that every homogeneous set of $H$ has size at most $(\log n)^{O(1)}$. 
	\end{claim}
	\begin{proof}
	It is enough to prove this in the case $|V(F)|=k+1$ or $F=C_5$, the 3-uniform tight cycle on 5 vertices. Indeed, every $F$ satisfying $|V(F)|\geq k+2$ induces some $F'\neq D_2$ with $|V(F')|=k+1$, with the exception of $F=C_5$. Let $G\sim G(n,1/2)$ be the random graph, where each edge is present independently with probability $1/2$, and let $H$ be the $k$-uniform hypergraph on $V(G)$, whose edges are the $k$-element cliques of $G$. Then $H$ has only polylogarithmic sized homogeneous sets. Note that if $k=3$, then $H$ contains no induced copy of $C_5$. Also, if $H$ contains an induced copy of a $k$-uniform hypergraph $F$ with $k+1$ vertices, then $|E(F)|\in \{0,1,2,k+1\}$. But then, if the complement of $H$ contains $F$, then $|E(F)|\in \{0,k-1,k,k+1\}$. Therefore, if neither $H$ nor the complement of $H$ avoids $F$, then  $F$ is the complete or empty $k$-uniform hypergraph on $k+1$ vertices, or $k=3$ and $F=D_2$. In the case $F$ is the complete or empty hypergraph, the claim is well known, see e.g. \cite{EHR65}. 
\end{proof}

Some further families of hypergraphs with polynomial homogeneous sets are studied in \cite{CS18, MS14, ST21}. Our paper is organized as follows. In the next subsections, we introduce our notation, and outline our proof. Then, in the next section, we prove several lemmas which will form the backbone of the proofs of our main theorems. We then prove Theorem \ref{thm:removal} in Section \ref{sect:removal} and Theorem \ref{thm:EH} in Section \nolinebreak \ref{sect:EH}. 

\subsection{Notation and preliminaries}
	For a graph $G$ and a set $X \subseteq V(G)$, the {\em density} of $X$ is $d(X) = \frac{e(X)}{\binom{|X|}{2}}$, where $e(X)$ is the number of edges of $G$ contained in $X$. For a pair of disjoint sets $X,Y \subseteq V(G)$, the density of $(X,Y)$ is $d(X,Y) = \frac{e(X,Y)}{|X||Y|}$. The pair $(X,Y)$ is called homogeneous if $d(X,Y) = 1$ or $d(X,Y) = 0$. We also use similar definitions for hypergraphs. For a $3$-uniform hypergraph $H$ and $X \subseteq V(H)$, the density of $X$ is $d(X) = \frac{e(X)}{\binom{|X|}{3}}$. For disjoint sets $X,Y,Z \subseteq V(H)$, the density of $(X,Y,Z)$ is $d(X,Y,Z) = \frac{e(X,Y,Z)}{|X||Y||Z|}$. Lastly, for disjoint sets $X,Y \subseteq V(H)$, we let $d(X,X,Y) = \frac{e(X,X,Y)}{\binom{|X|}{2}|Y|}$, where $e(X,X,Y)$ is the number of edges which have two vertices in $X$ and one in $Y$. We say that $(X,Y)$ is {\em $\varepsilon$-homogeneous} if either $d(X,X,Y),d(Y,Y,X) \geq 1 - \varepsilon$ or $d(X,X,Y),d(Y,Y,X) \leq \varepsilon$. If $(X,Y)$ is $0$-homogeneous then we simply say that $(X,Y)$ is {\em homogeneous}.
	
	Recall that a graph $G$ is called a {\em cograph} if either $|V(G)| = 1$ or $G$ can be obtained from two smaller vertex-disjoint cographs $G_1,G_2$ by placing a complete or empty bipartite graph between $V(G_1)$ and $V(G_2)$. It is well-known that cographs are perfect, implying that every cograph on $n$ vertices contains a clique or an independent set of size at least $\sqrt{n}$.
	It is also well-known \cite{Seinsche} that a graph $G$ is a cograph if and only if it is induced $P_4$-free, where $P_4$ is the path with 4 vertices. Alon and Fox \cite{AF} proved a polynomial removal lemma for cographs:
	\begin{theorem}[\cite{AF}]\label{lem:P_4}
		For every $\zeta > 0$ there is $\delta = \delta_{{\text{cograph}}}(\zeta) = \poly(\zeta) > 0$ such that if an $n$-vertex graph $G$ contains at most $\delta n^4$ induced copies of $P_4$, then $G$ can be turned into a cograph by adding/deleting at most $\zeta n^2$ edges. 
	\end{theorem}

\subsection{Outline of the proofs}

Let us give a rough outline of the proofs of our main theorems. In order to prove Theorems \ref{thm:removal} and \ref{thm:EH}, we establish certain structural results about 3-graphs with no (or few) induced copies of $D_2$. The first crucial observation is that if $H$ is $D_2$-free, then the link graph of every vertex is a cograph. The bulk of the work is then put into proving the following: if $|V(H)|=n$ and $H$ contains (only) $\delta n^{4}$ induced copies of $D_2$, then there is a partition $V(H)=\mathcal{X}\cup\mathcal{Y}\cup\mathcal{S}$ such that $|\mathcal{S}|\leq \xi n$  (which we view as a set of `leftovers'), $|\mathcal{X}|,|\mathcal{Y}|>\xi n$, and $(\mathcal{X},\mathcal{Y})$ is $\varepsilon$-homogeneous, provided that $\delta \leq (\xi\varepsilon)^{c}$ for some absolute constant $c>0$. This is essentially done in Lemma \ref{lem:main}. By recursively applying this partition result, we arrive to a partition $V(H)=\mathcal{X}_1\cup \dots\cup \mathcal{X}_k\cup \mathcal{S}'$ and a cograph $C$ on vertex set $[k]$ with the following properties: $k=(1/\xi)^{O(1)}$, $|\mathcal{S}'|$ is small, $|\mathcal{X}_1|,\dots,|\mathcal{X}_k| \geq \xi^{O(1)}n$, $d(\mathcal{X}_i,\mathcal{X}_i,\mathcal{X}_j),d(\mathcal{X}_j,\mathcal{X}_i,\mathcal{X}_i)\geq 1-\varepsilon$ if $ij\in E(C)$, and $d(\mathcal{X}_i,\mathcal{X}_i,\mathcal{X}_j),d(\mathcal{X}_j,\mathcal{X}_i,\mathcal{X}_i) \leq \varepsilon$ if $ij\not\in E(C)$. Controlling the densities of pairs $(\mathcal{X}_i,\mathcal{X}_j)$ and using the fact that the hypergraph has few induced $D_2$'s then allows us to also control the densities of triples $(\mathcal{X}_i,\mathcal{X}_j,\mathcal{X}_k)$.

In order to prove Theorem \ref{thm:removal}, we do some cleaning up with the help of this partition. To prove Theorem \ref{thm:EH}, we observe that we can choose $\xi,\varepsilon=n^{-\Omega(1)}$, and show that a large clique or independent set in $C$ corresponds to a large clique or independent set in $H$.
	
	\section{The Main Lemmas}\label{sect:lemmas}

	In this section, we prove some lemmas regarding the structure of hypergraphs with no (or few) induced copies of $D_2$. 
	For a $3$-uniform hypergraph $H$ and $v \in V(H)$, denote by $L(v)$ the link graph of $v$. 
	
	\begin{lemma}\label{lem:cograph_links}
		If a $3$-uniform hypergraph $H$ is induced $D_2$-free, then all link graphs in $H$ are cographs. 
	\end{lemma}
	\begin{proof}
		Let $v \in V(H)$, and suppose by contradiction that $a,b,c,d$ is an induced path in $L(v)$. In order for $v,a,b,c$ not to span an induced copy of $D_2$, we must have $\{a,b,c\} \in E(H)$. Similarly, $\{b,c,d\} \in E(H)$, and also,  $\{a,b,d\},\{a,c,d\} \notin E(H)$. But now $a,b,c,d$ span an induced copy of $D_2$, a contradiction. 
	\end{proof}

	\begin{lemma}\label{lem:pairs_to_triple}
		Let $\varepsilon > 0$, let $H$ be a $3$-uniform hypergraph and let $X,Y,Z \subseteq V(H)$ be disjoint, $|X| \geq 2/\varepsilon$. If $d(X,X,Y),d(X,X,Z) \geq 1 - \frac{\varepsilon^2}{8}$, then one of the following holds:
		\begin{enumerate}
			\item  $d(X,Y,Z) \geq 1 - \varepsilon$,
			\item $H$ contains at least $\frac{\varepsilon^2}{16}|X|^2|Y||Z|$ induced copies of $D_2$.
		\end{enumerate}
		Similarly, if $d(X,X,Y),d(X,X,Z) \leq \frac{\varepsilon^2}{8}$, then one of the following holds:
			\begin{enumerate}
			\item $d(X,Y,Z) \leq \varepsilon$.
			\item $H$ contains at least $\frac{\varepsilon^2}{16}|X|^2|Y||Z|$ induced copies of $D_2$.
		\end{enumerate}
	\end{lemma}
	\begin{proof}
		Suppose that $d(X,X,Y),d(X,X,Z) \geq 1 - \frac{\varepsilon^2}{8}$, the other case follows by taking complements. Assume that Item 1 does not hold, that is, $d(X,Y,Z) < 1 - \varepsilon$. We will show that Item 2 holds. For each $y \in Y, z \in Z$, let $d(y,z)$ be the number of $x \in X$ such that $\{x,y,z\} \notin E(H)$. Then $\sum_{(y,z) \in Y \times Z}{d(y,z)} > \varepsilon |X||Y||Z|$. By Jensen's inequality, we have
		\begin{equation}\label{eq:d(X,Y,Z) Jensen}
		\sum_{(y,z) \in Y \times Z}{\binom{d(y,z)}{2}} > |Y||Z| \cdot \binom{\varepsilon |X|}{2} \geq \varepsilon^2|Y||Z||X|^2/4 \geq \frac{\varepsilon^2}{2}\binom{|X|}{2}|Y||Z|. 
		\end{equation}
		Sample $\{x,x'\} \in \binom{X}{2}$, $y \in Y$, $z \in Z$  independently and uniformly at random. By \eqref{eq:d(X,Y,Z) Jensen}, we have that  $\{x,y,z\},\{x',y,z\} \notin E(H)$ holds with probability at least $\frac{\varepsilon^2}{2}$. On the other hand, since $d(X,X,Y),d(X,X,Z) \geq 1-\frac{\varepsilon^2}{8}$, the probability that $\{x,x',y\} \notin E(H)$ or $\{x,x',z\} \notin E(H)$ is at most $2 \cdot \frac{\varepsilon^2}{8} = \frac{\varepsilon^2}{4}$. Hence, with probability at least $\frac{\varepsilon^2}{4}$, the vertices $x,x',y,z$ span an induced copy of $D_2$. This means that $H$ contains at least $\frac{\varepsilon^2}{4} \binom{|X|}{2}|Y||Z| \geq \frac{\varepsilon^2}{16}|X|^2|Y||Z|$ induced copies of $D_2$, as required by Item 2. 
	\end{proof}

\begin{lemma}\label{lem:homogeneous}
	Let $\gamma > 0$, let $H$ be an $n$-vertex $3$-uniform hypergraph, let $v \in V(H)$, and let $X,Y,Z \subseteq V(H) \setminus \{v\}$ be disjoint, $|X|,|Y| \geq 2/\gamma$.
	Suppose that $d_{L(v)}(X,Z) \geq 1 - \frac{\gamma^2}{80}$ and $d_{L(v)}(Y,Z) \leq \frac{\gamma^2}{80}$. If $d_{L(v)}(X,Y) \geq 1 - \frac{\gamma^2}{80}$, then one of the following holds:
	\begin{enumerate}
		\item  $d(X,X,Y),d(Y,Y,X) \geq 1 - \gamma$, 
		\item $H$ contains at least 
		$\frac{\gamma^2}{128} \cdot \min \left\{ 
		|X|^2 |Y||Z|, \;
		|X| |Y|^2|Z|, \; |X|^2 |Y|^2
		\right\}$
		induced copies of $D_2$. 
		\item $H$ contains at least 
		$\frac{\gamma^2}{128} \cdot \frac{1}{n} \cdot 
		\min \left\{|X|^2 |Y||Z|, \; |X| |Y|^2|Z| \right\}$ 
		induced copies of $D_2$ which contain \nolinebreak $v$. 
	\end{enumerate} 
    Similarly, if $d_{L(v)}(X,Y) \leq \frac{\gamma^2}{80}$, then one of the following holds:
    	\begin{enumerate}
		\item $d(X,X,Y),d(Y,Y,X) \leq \gamma$.
		\item $H$ contains at least 
		$\frac{\gamma^2}{128} \cdot \min \left\{ 
		|X|^2 |Y||Z|, \;
		|X| |Y|^2|Z|, \; |X|^2 |Y|^2
		\right\}$
		induced copies of $D_2$. 
		\item $H$ contains at least 
		$\frac{\gamma^2}{128} \cdot \frac{1}{n} \cdot 
		\min \left\{|X|^2 |Y||Z|, \; |X| |Y|^2|Z| \right\}$ 
		induced copies of $D_2$ which contain \nolinebreak $v$. 
	\end{enumerate} 
\end{lemma}
\begin{proof}
	 We will only prove the assertion in the case that $d_{L(v)}(X,Y) \geq 1 - \frac{\gamma^2}{80}$; 
	the case $d_{L(v)}(X,Y) \leq \frac{\gamma^2}{80}$ then follows by considering the complement of $H$ and switching the roles of $X$ and $Y$. We assume Items 2,3 do not hold, and show that then Item 1 must hold. Let us first show that $d(Y,Y,X) \geq 1 - \frac{\gamma^2}{8} \geq 1 - \gamma$. So suppose for contradiction that 
	$d(Y,Y,X) < 1 - \frac{\gamma^2}{8}$. Sample 
	$x \in X, \{y,y'\} \in \binom{Y}{2}, z \in Z$ uniformly at random and independently. Let $\mathcal{A}$ be the event that the following two items hold:
	\begin{enumerate}
		\item[(a)] $\{x,y,y'\} \notin E(H)$.
		\item[(b)] $\{x,y\},\{x,y'\},\{x,z\} \in E(L(v))$ and $\{z,y\},\{z,y'\} \notin E(L(v))$.
	\end{enumerate}
	By assumption, (a) holds with probability larger than $\frac{\gamma^2}{8}$. Also, since $d_{L(v)}(X,Y), d_{L(v)}(X,Z) \geq 1 - \frac{\gamma^2}{80}$ and $d_{L(v)}(Y,Z) \leq \frac{\gamma^2}{80}$, Item (b) holds with probability at least $1 - 5 \cdot \frac{\gamma^2}{80} = 1 - \frac{\gamma^2}{16}$. Hence, $\mathbb{P}[\mathcal{A}] \geq \frac{\gamma^2}{16}$.
	Therefore, the number of 4-tuples $(x,\{y,y'\},z)$ satisfying $\mathcal{A}$ is at least 
	$\frac{\gamma^2}{16} |X| \binom{|Y|}{2}|Z| \geq \frac{\gamma^2}{64}|X||Y|^2|Z|$. 
	We will now show that if $\mathcal{A}$ happens then $H[\{v,x,y,y',z\}]$ contains an induced copy of $D_2$. This would imply that either $H$ contains at least $\frac{\gamma^2}{128}|X||Y|^2|Z|$ induced copies of $D_2$, or $H$ contains at least $\frac{\gamma^2}{128} \cdot \frac{1}{n} \cdot |X||Y|^2|Z|$ induced copies of $D_2$ which contain $v$. In either case, we will obtain a contradiction to the assumption that Items 2 and 3 do not hold. 
	
	So suppose that $\mathcal{A}$ happens. Note that 
	\begin{itemize}
	    \item $\{v,x,y\},\{v,x,y'\},\{v,x,z\} \in E(H)$,
	    \item $\{x,y,y'\},\{v,z,y\},\{v,z,y'\} \notin E(H)$.
	\end{itemize}
  Now consider several cases. If $\{x,y,z\} \notin E(H)$ then $v,x,y,z$ span an induced copy of $D_2$. So we may assume that $\{x,y,z\} \in E(H)$, and similarly $\{x,y',z\} \in E(H)$. If $\{v,y,y'\} \notin E(H)$ then $v,x,y,y'$ span an induced copy of $D_2$, so we may assume that $\{v,y,y'\} \in E(H)$. Finally, if $\{y,y',z\}\in E(H)$, then $v,y,y',z$ span an induced copy of $D_2$, and if $\{y,y',z\} \notin E(H)$, then $x,y,y',z$ span an induced copy of $D_2$, proving our claim. 
	
	So far we have shown that $d(Y,Y,X) \geq 1 - \frac{\gamma^2}{8}$. Let us now show that $d(X,X,Y) \geq 1 - \gamma$. Suppose for a contradiction that $d(X,X,Y) < 1 - \gamma$. For a pair $\{x,x'\} \in \binom{X}{2}$, let $d(x,x')$ be the number of $y \in Y$ such that $\{x,x',y\} \notin E(H)$. Then $\sum_{\{x,x'\}}{d(x,x')} > \gamma\binom{|X|}{2}|Y|$. By Jensen's inequality, 
	$$
	\sum_{\{x,x'\}}{\binom{d(x,x')}{2}} > \binom{|X|}{2}\binom{\gamma |Y|}{2} \geq \frac{\gamma^2}{2}\binom{|X|}{2}\binom{|Y|}{2}.
	$$
	This means that when sampling $\{x,x'\} \in \binom{X}{2}, \{y,y'\} \in \binom{Y}{2}$ uniformly at random, we have \linebreak $\{x,x',y\},\{x,x',y'\} \notin E(H)$ with probability at least $\frac{\gamma^2}{2}$. Also, since $d(Y,Y,X) \geq 1 - \frac{\gamma^2}{8}$, we have that $\{x,y,y'\},\{x',y,y'\} \in E(H)$ with probability at least $1 - \frac{\gamma^2}{4}$. Hence, with probability at least $\frac{\gamma^2}{4}$, the vertices $x,x',y,y'$ induce a copy of $D_2$. Therefore, $H$ has at least 
	$\frac{\gamma^2}{4}\binom{|X|}{2}\binom{|Y|}{2} \geq \frac{\gamma^2}{64}|X|^2|Y|^2$ 
	induced copies of $D_2$, a contradiction. 
\end{proof}

	Next, in the following lemma, we show that cographs have partitions in which all pairs of parts, except for pairs which form a matching, are homogeneous. We will also guarantee some additional structure on the non-homogeneous pairs.

\begin{lemma}\label{lem:cograph_partition_main}
		Let $1 \leq m < n$ be integers, let $0 < \beta < 1$,
		and let $G$ be a cograph on $n$ vertices. Then there is a partition $V(G) = S \cup V_1 \cup \dots \cup V_t$ and a set of pairs $\mathcal{M} \subseteq \binom{[t]}{2}$ forming a matching, such that the following holds:
		\begin{enumerate}
			\item $|S| \leq (2\lceil \frac{n}{m} \rceil - 3) \cdot 10\beta m \leq 20\beta n$.
			\item $\beta m \leq |V_i| \leq m$ for every $1 \leq i \leq t$. 
			\item For every $1 \leq i < j \leq t$, if $\{i,j\} \notin \mathcal{M}$ then the pair $(V_i,V_j)$ is homogeneous. 
			\item For every $\{i,j\} \in \mathcal{M}$, there are $V'_i \subseteq V_i, V'_j \subseteq V_j$ such that 
			$|V'_i|,|V'_j| \geq \beta^3 m/2$
			and one of the following holds:
			\begin{enumerate}
				\item $V'_i$ is complete to $V_j$ and $V'_j$ is complete to $V_i$.
				\item $V'_i$ is empty to $V_j$ and $V'_j$ is empty to $V_i$.
			\end{enumerate} 
		\end{enumerate}
	\end{lemma}
	\begin{proof}
		The proof is by induction on $n$. The base case $n = 2$ is trivial. 
		Given $G$, define sets $A_1,A_2,\dots$ as follows. 
		For each $i \geq 0$, if 
		$|A_1 \cup \dots \cup A_i| \geq m$ then stop. Otherwise let $A \cup B$ be a partition of $V(G) \setminus (A_1 \cup \dots \cup A_i)$ such that the bipartite graph $(A,B)$ is complete or empty and $|A| \leq |B|$. Such a partition is guaranteed because $G$ is a cograph, as $|V(G) \setminus (A_1 \cup \dots \cup A_i)| > n - m \geq 1$.
		Set $A_{i+1} = A$. Suppose that the process stopped after $k$ steps, and consider the sets $A_1,\dots,A_k$. By definition, we have $|A_1 \cup \dots \cup A_{k-1}| < m$ and $|A_1 \cup \dots \cup A_k| \geq m$. For each $1 \leq i \leq k$, the bipartite graph between $A_i$ and $V(G) \setminus (A_1 \cup \dots \cup A_i)$ is either complete or empty. Let $I^+$ be the set of indices $1 \leq i \leq k-1$ for which this bipartite graph is complete, and $I^-$ the set of $1 \leq i \leq k-1$ for which this bipartite graph is empty. 
		
		Put $X := \bigcup_{i \in I^+}{A_i}$, $Y := \bigcup_{i \in I^-}{A_i}$, 
		$Z_1 = A_k$ and $Z_2 = V(G) \setminus (A_1 \cup \dots \cup A_k) = V(G) \setminus (X \cup Y \cup Z_1)$. Note that $X$ is complete to $Z_1 \cup Z_2$ and $Y$ is empty to $Z_1 \cup Z_2$. Also, the bipartite graph $(Z_1,Z_2)$ is homogeneous.  
		Set $r = \lceil \frac{1}{\beta} \rceil$. Order the elements of $X \cup Y = A_1 \cup \dots \cup A_{k-1}$ such that the elements of $A_i$ come before the elements of $A_j$ for every $1 \leq i < j \leq k-1$, and partition $X \cup Y$ into $r$ intervals $I_1,\dots,I_r$ according to this order, 
		with $|I_i| = \lfloor |X \cup Y|/r \rfloor$ or $|I_i| = \lceil |X \cup Y|/r \rceil$ for every $i$, 
		and with $|I_1| \geq \dots \geq |I_r|$. 
		Observe that for $1 \leq i < j \leq r$, the bipartite graph $(X \cap I_i, Y \cap I_j)$ is complete, and the bipartite graph $(Y \cap I_i, X \cap I_j)$ is empty. Suppose without loss of generality that $|X \cap I_1| \geq |I_1|/2$; the case that $|Y \cap I_1| \geq |I_1|/2$ is symmetric. Let $s$ be the largest integer satisfying $|Y \cap I_s| > \beta |I_s|$; if no such $s$ exists then set $s = 0$. Set $X_1 = X \cap (I_1 \cup \dots \cup I_{s-1})$, 
		$Y_1 = Y \cap (I_2 \cup \dots \cup I_s)$, $X_2 = X \cap (I_{s+1} \cup \dots \cup I_r)$. Initialize $S = \emptyset$, and put the elements of
		$Y \cap (I_{s+1} \cup \dots \cup I_r)$, $Y \cap I_1$, and $X \cap I_s$ into $S$. Then $X \cup Y = X_1 \cup X_2 \cup Y_1 \cup S$. 
		Note that $|Y \cap (I_{s+1} \cup \dots \cup I_r)| \leq \beta (|I_{s+1}| + \dots + |I_r|) \leq \beta (|X| + |Y|) < \beta m$ by our choice of $s$. 
		We claim that $|Y \cap I_1|,|X \cap I_s| \leq 2\beta m$. 
		Recall that $|I_1|,|I_s| \leq \lceil |X \cup Y|/r \rceil$. So if $|X \cup Y| \geq r$, then we have 
		$|I_1|,|I_s| \leq \lceil |X \cup Y|/r \rceil \leq 2|X \cup Y|/r \leq \frac{2m}{r} \leq 2\beta m$. Now suppose that $|X \cup Y| < r$. Then $|I_1|,|I_s| \leq 1$. As $|Y \cap I_1| \leq |I_1|/2$, it must be that $Y \cap I_1 = \emptyset$. Similarly, as $|Y \cap I_s| > \beta|I_s|$, it must be that $I_s \subseteq Y$, and hence $X \cap I_s = \emptyset$. So in both cases, $|Y \cap I_1|,|X \cap I_s| \leq 2\beta m$. 
		It follows that after adding the above sets to $S$, we have $|S| < 5\beta m$.
		
		For each of the sets $X_1,X_2,Y_1$, if it has size less than $\beta m$ then put its elements into $S$. After this step, we have $|S| < 8\beta m$. 
		
		Note that the bipartite graph $(Y_1,X_2)$ is empty. Also, the bipartite graph $(X_1,X_2)$ is complete because $Y \cap I_s \neq \emptyset$, implying that there is no $i \in I^+$ with $A_i \cap X_1, A_i \cap X_2 \neq \emptyset$. 
		Setting $X'_1 := X \cap I_1 \subseteq X_1$ and $Y'_1 := Y \cap I_s \subseteq Y_1$, observe that $X'_1$ is complete to $Y_1$ and $Y'_1$ is complete to $X_1$. 
		
		Suppose that $|X_1|,|Y_1| \geq \beta m$. 
		Then we have $|X'_1| \geq |I_1|/2 \geq \frac{|X \cup Y|}{2r} \geq 
		\frac{2\beta m}{2r} \geq \beta^2 m/2$. Also,
		$s \geq 2$ (because else $Y_1$ is empty) and $|Y'_1| > \beta |I_s|$, hence $I_s \neq \emptyset$. 
		So we have
		$|Y'_1| > \beta |I_s| \geq
		\frac{\beta|X \cup Y|}{2r} \geq \frac{\beta^2 m}{r} \geq 
		\beta^3 m/2$, where the second inequality holds because $I_s \neq \emptyset$, so that either $|I_s| = \lceil |X \cup Y|/r \rceil$ or $|I_s| = \lfloor |X \cup Y|/r \rfloor \geq 1$.
		
		For each $i = 1,2$, we do the following. If $|Z_i| < \beta m$ then place $Z_i$ into $S$. Following this step, $|S| < 10\beta m$. If $|Z_i| > m$ then apply the induction hypothesis to 
		$G[Z_i]$ (with the same parameters $m$ and $\beta$) to obtain a partition $Z_i = S_i \cup U_{i,1} \cup \dots \cup U_{i,t_i}$ and a matching $\mathcal{M}_i \subseteq \binom{[t_i]}{2}$ satisfying Items 1-4. In particular, we have
		$|S_i| \leq 
		(2\big\lceil \frac{|Z_i|}{m} \big\rceil - 3) \cdot 10\beta m$. 
		Hence 
		$$|S_1| + |S_2| \leq \left(2\left\lceil \frac{|Z_1|}{m} \right\rceil + 2\left\lceil \frac{|Z_2|}{m} \right\rceil - 6\right) \cdot 10\beta m \leq 
		\left(2\left\lceil \frac{n}{m} \right\rceil - 4\right) \cdot 10\beta m,$$
		where the second inequality holds by the general inequality $\lceil x\rceil+\lceil y\rceil\leq \lceil x+y\rceil +1$.
		Add all elements of $S_1 \cup S_2$ into $S$. We now have 
		$|S| \leq (2\lceil \frac{n}{m} \rceil - 3) \cdot 10\beta m$, as required. 
		
		Our partition of $V(G) \setminus S$ consists of the parts 
		$U_{i,1} ,\dots, U_{i,t_i}$ (for $i = 1,2$ for which $|Z_i| > m$); of $Z_i$ for $i = 1,2$ with $\beta m \leq |Z_i| \leq m$; and of those sets among $X_1,Y_1,X_2$ which were not placed into $S$. 
		The matching $\mathcal{M}$ consists of $\mathcal{M}_1 \cup \mathcal{M}_2$ and the edge $\{X_1,Y_1\}$ (unless one of $X_1,Y_1$ was placed into $S$). It is easy to see that all requirements in Items 1-4 are satisfied. 
	\end{proof}

		In the following lemma we will consider a graph $G$ with a weight function $w : V(G) \rightarrow \mathbb{R}_{> 0}$. For a subset $A \subseteq V(G)$, we use the notation $w(A) = \sum_{v \in A}{w(v)}$.
	
	\begin{lemma}\label{lem:weighted_cograph_partition}
		Let $0 < \beta \leq 1/3$, let $G$ be a cograph and let $w : V(G) \rightarrow \mathbb{R}_{> 0}$ such that $\sum_{v \in V(G)}{w(v)} = 1$ and $w(v) \leq 1-\beta$ for every $v \in V(G)$. Then there is a partition $V(G) = \mathcal{I} \cup \mathcal{J} \cup \mathcal{L}$ such that $w(\mathcal{I}),w(\mathcal{J}) \geq \beta/2$, $w(\mathcal{L}) < \beta$, and the bipartite graph between $\mathcal{I}$ and $\mathcal{J}$ is complete or empty. 
	\end{lemma}
	\begin{proof}
		Define sets $A_1,A_2,\dots$ as follows. 
		For each $i \geq 0$, if 
		$w(A_1 \cup \dots \cup A_i) \geq \beta$ then stop. Otherwise let $A \cup B$ be a partition of $V(G) \setminus (A_1 \cup \dots \cup A_i)$ such that the bipartite graph $(A,B)$ is complete or empty and $w(A) \leq w(B)$; set $A_{i+1} = A$.  Such a partition $(A,B)$ exists because if $w(A_1 \cup \dots \cup A_i) < \beta$, then there are at least two vertices outside of $A_1 \cup \dots \cup A_i$ (since $w(v) \leq 1-\beta$ for every $v \in V(G)$). This process has to stop at some point. 
		Suppose that the process stopped after $k$ steps, and consider the sets $A_1,\dots,A_k$. By definition, we have 
		$w(A_1 \cup \dots \cup A_{k-1}) < \beta$ and 
		$w(A_1 \cup \dots \cup A_k) \geq \beta$. For each $1 \leq i \leq k$, the bipartite graph between $A_i$ and $V(G) \setminus (A_1 \cup \dots \cup A_i)$ is either complete or empty. 
		
		We consider two cases. If $w(A_k) \geq \beta$, then $\mathcal{I} = A_k, \mathcal{J} = V(G) \setminus (A_1 \cup \dots \cup A_k), \mathcal{L} = A_1 \cup \dots \cup A_{k-1}$ satisfy the assertion of the lemma. Note that $w(\mathcal{J}) \geq w(\mathcal{I})$ by construction. Suppose now that $w(A_k) < \beta$. Then $w(A_1 \cup \dots \cup A_k) < 2\beta$. Let $I^+$ be the set of $1 \leq i \leq k$ for which the bipartite graph between $A_i$ and $V(G) \setminus (A_1 \cup \dots \cup A_i)$ is complete, and let $I^-$ be the set of $1 \leq i \leq k$ for which this bipartite graph is empty. Without loss of generality, $w(\bigcup_{i \in I^+}{A_i}) \geq 
		w(\bigcup_{i \in I^-}{A_i})$. Now set $\mathcal{I} = \bigcup_{i \in I^+}{A_i}$, $\mathcal{J} = V(G) \setminus (A_1 \cup \dots \cup A_k)$ and $\mathcal{L} = \bigcup_{i \in I^-}{A_i}$. Then $w(\mathcal{I}) \geq w(A_1 \cup \dots \cup A_k)/2 \geq \beta/2$, $w(\mathcal{J}) > 1 - 2\beta \geq \beta$ and $w(\mathcal{L}) \leq w(A_1 \cup \dots \cup A_k)/2 < \beta$. Also, the bipartite graph between $\mathcal{I}$ and $\mathcal{J}$ is complete. 
	\end{proof}
	\noindent
	The following is the main lemma of the paper, and will be used in the proofs of Theorems \ref{thm:removal} and \nolinebreak \ref{thm:EH}. 
		\begin{lemma}\label{lem:main}
		For every $\xi,\varepsilon > 0$ there are 
		$\delta = \delta(\xi,\varepsilon) = \poly(\xi\varepsilon) > 0$ and 
		$n_{\ref{lem:main}} = n_{\ref{lem:main}}(\xi,\varepsilon) = \poly(\frac{1}{\xi\varepsilon})$  
		such that the following holds. Let $H$ be a $3$-uniform hypergraph on $n \geq n_{\ref{lem:main}}$ vertices, let $v \in V(H)$, and suppose that 
		\begin{enumerate}
		    \item $\xi \leq d(L(v)) \leq 1 - \xi$,
		    \item  $H$ contains less than $\delta n^4$ induced copies of $D_2$,
		    \item $H$ contains less than $\delta n^3$ induced copies of $D_2$ containing $v$.
		\end{enumerate}
		  Then there is a partition $V(H) = \mathcal{X} \cup \mathcal{Y} \cup \mathcal{S}$ such that
		$|\mathcal{X}|,|\mathcal{Y}| \geq \frac{\xi}{10}n$, $|\mathcal{S}| \leq \frac{\xi}{2}n$ and $(\mathcal{X},\mathcal{Y})$ is $\varepsilon$-homogeneous. 
	\end{lemma}
	As the proof of Lemma \ref{lem:main} is somewhat technical, let us give a rough outline. Fist, we apply Theorem \ref{lem:P_4} to $L(v)$ to turn it into a cograph $G$ using few edge additions/deletions. This is possible because $L(v)$ contains few induced $P_4$'s because of Lemma \ref{lem:cograph_links} and Items 2-3 in Lemma \ref{lem:main}.
	Then we apply Lemma \ref{lem:cograph_partition_main} to the cograph $G$ to obtain the partition $S \cup V_1 \cup \dots \cup V_t$. We consider the reduced graph on $[t]$; in this graph, $\{i,j\}$ is an edge if $(V_i,V_j)$ is a complete bipartite graph, and a non-edge if it is an empty bipartite graph. For pairs $\{i,j\} \in \mathcal{M}$, we use Item 4 in Lemma \ref{lem:cograph_partition_main} to define the adjacency. We group $1,\dots,t$ into groups $I_1,\dots,I_r$ such that if $i,j$ are in the same group then they have the same relation to all other vertices in the reduced graph. 
	Each group $I_a$ forms a clique or an independent set in the reduced graph. Consequently, no single group $I_a$ can span almost all of the vertices of $G$, because otherwise $G$ (and hence also $L(v)$) would have density very close to $0$ or $1$, contradicting Item 1 in the lemma. 
	If $i,j$ belong to different groups, then there is some $k$ such that $i,j$ relate differently to $k$, and then we can use Lemma \ref{lem:homogeneous} to deduce that $(V_i,V_j)$ is $\gamma$-homogeneous for an appropriate $\gamma$. (This argument does not apply if $\{i,j\} \in \mathcal{M}$, but there are very few such pairs, and so their contribution is negligible.). Then, using Lemma \ref{lem:pairs_to_triple}, we can deduce that almost all triples $(V_i,V_j,V_k)$ are $\varepsilon$-homogeneous. This already gives us a lot of structure on the hypergraph $H$. As a final step, we group $V_1,\dots,V_t$ into just two groups, which will correspond to the sets $\mathcal{X},\mathcal{Y}$ in the statement of Lemma \ref{lem:main}, by applying Lemma \ref{lem:weighted_cograph_partition} to the reduced graph (or, more precisely, to the graph on $I_1,\dots,I_r$ derived from the reduced graph). The full proof follows. 
	\begin{proof}[Proof of Lemma \ref{lem:main}]
		Set 
		$$
		\alpha:= \frac{\xi^2\varepsilon}{800}, 
		\; \; \; 
		\beta:= \frac{\xi}{240},
		\; \; \;
		\gamma:= \frac{\varepsilon^2}{32},
		\; \; \;
		\zeta:= \frac{\alpha^2\beta^4\gamma^2}{160},
		$$
		$$
		\delta:= \min\left\{ \frac{1}{4} \cdot \delta_{\text{cograph}}(\zeta), \;
		\frac{\gamma^2}{128} \cdot \frac{\alpha^4\beta^{6}}{4}, \;
		\frac{\varepsilon^2}{128} \cdot (\alpha\beta)^4 
		\right\}.
		$$
		Here, $\delta_{\text{cograph}}$ is from Theorem \ref{lem:P_4}. 
		
		We claim that $L(v)$ has at most $4 \delta (n-1)^4$ induced copies of $P_4$. Indeed, for each copy $X$ of $P_4$, it holds by Lemma \ref{lem:cograph_links} that $X \cup \{v\}$ contains an induced copy of $D_2$. Hence, either $X$ spans an induced copy of $D_2$, or there is an induced copy of $D_2$ consisting of $v$ and $3$ vertices from $X$. If at least half of the copies $X$ of $P_4$ are of the first kind, then we get $2\delta (n-1)^4 \geq \delta n^4$ induced copies of $D_2$ in $H$, a contradiction. And if at least half are of the second kind, then we get at least $2\delta (n-1)^3 \geq \delta n^3$ induced copies of $D_2$ containing $v$, again giving a contradiction. 
		
		By our choice of $\delta$ via Theorem \ref{lem:P_4}, we get that $L(v)$ can be turned into a cograph $G$ by adding/deleting at most $\zeta (n-1)^2$ edges. Apply Lemma \ref{lem:cograph_partition_main} to $G$ with $m := \lceil \alpha (n-1) \rceil$ and $\beta$ as above, to obtain a partition $V(G) = S \cup V_1 \cup \dots \cup V_t$ and a matching $\mathcal{M} \subseteq \binom{[t]}{2}$ satisfying Items 1-4 in that lemma. 
		Define a graph $K$ on $[t]$ as follows. For $1 \leq i < j \leq t$ with $\{i,j\} \notin \mathcal{M}$, let $\{i,j\} \in E(K)$ if the bipartite graph $(V_i,V_j)$ is complete and $\{i,j\} \notin E(K)$ if the bipartite graph $(V_i,V_j)$ is empty. By Item 3 in Lemma \ref{lem:cograph_partition_main}, one of these options holds. Suppose now that $\{i,j\} \in \mathcal{M}$, and let $V'_i \subseteq V_i, V'_j \subseteq V_j$ be as in Item 4 of Lemma \ref{lem:cograph_partition_main}. If Item 4(a) in Lemma \ref{lem:cograph_partition_main} holds then let $\{i,j\} \in E(K)$, and if Item 4(b) holds then let $\{i,j\} \notin E(K)$. 
		
		Define the relation $\sim$ on $[t]$ by letting $i\sim j$ if and only if for every $k \in [t] \setminus \{i,j\}$ it holds that $\{i,k\} \in E(K)$ if and only if $\{j,k\} \in E(K)$. Then $\sim$ is an equivalence relation. Let $I_1,\dots,I_r$ be the equivalence classes of $\sim$; so $[t] = I_1 \cup \dots \cup I_r$. 
		Note that for every $1 \leq a < b \leq r$, the bipartite graph between $I_a$ and $I_b$ in $K$ is complete or empty.  
		Let $F$ be the corresponding reduced graph on $[r]$; that is, $\{a,b\} \in E(F)$ if the bipartite graph between $I_a$ and $I_b$ is complete and $\{a,b\} \notin E(F)$ if this bipartite graph is empty.
		
		Let us define sets $U_i \subseteq V_i$, $i \in [t]$, as follows. If there is $j$ such that $\{i,j\} \in \mathcal{M}$, then take $U_i$ to be the set $V'_i$ from Item 4 in Lemma \ref{lem:cograph_partition_main}, and otherwise take $U_i = V_i$. 
		Observe that by the definition of $K$ and $F$, for all $1 \leq a \neq b \leq r$ and $i \in I_a, j \in I_b$ we have that $(U_i,V_j)$ is complete in $G$ if $\{a,b\} \in E(F)$ and $(U_i,V_j)$ is empty in $G$ if $\{a,b\} \notin E(F)$. Moreover, if $\{i,j\} \notin \mathcal{M}$ then the same is true for $(V_i,V_j)$; namely, $(V_i,V_j)$ is complete in $G$ if $\{a,b\} \in E(F)$ and $(V_i,V_j)$ is empty in $G$ if $\{a,b\} \notin E(F)$.
		Note also that by Items 2 and 4 in Lemma \ref{lem:cograph_partition_main}, and as $m \geq \alpha (n-1)$, we have
		\begin{equation}\label{eq:V_i U_i}
		|V_i| \geq \alpha\beta(n-1), 
		\; \; \;
		|U_i| \geq \alpha\beta^3 (n-1)/2
		\end{equation}
		for every $i \in [t]$. 
		
		\begin{claim}\label{claim:cograph_pattern}
			$F$ is a cograph.
		\end{claim}
		\begin{proof}
			Suppose, for the sake of contradiction, that $(a,b,c,d)$ is an induced path in $F$ for some $a,b,c,d \in [r]$. Fix $i \in I_a, j \in I_b, k \in I_c, \ell \in I_d$. We saw above that the bipartite graphs $(U_i,U_j),(U_j,U_k),(U_k,U_{\ell})$ are complete in $G$ and the bipartite graphs $(U_i,U_k),(U_j,U_{\ell}),(U_i,U_{\ell})$ are empty in $G$. It follows that $G$ contains an induced path on 4 vertices (every quadruple in $U_i \times U_j \times U_k \times U_{\ell}$ spans such a path), in contradiction to $G$ being a cograph.  
		\end{proof} 
		\noindent
		For $a \in [r]$, let $W_a := \bigcup_{i \in I_a}{V_i}$. Note that $V(G) = W_1 \cup \dots \cup W_r \cup S$.
		\begin{claim}\label{claim:non-trivial_partition}
			For every $1 \leq a \leq r$, $|W_a| \leq (1 - \frac{\xi}{3})(n-1)$. 
		\end{claim}
		\begin{proof}
			Suppose, for the sake of contradiction, that $|W_a| \geq (1-\frac{\xi}{3})(n-1)$. 
			Observe that $I_a$ is either a clique or an independent set in $K$. Let us assume that it is a clique (the other case can be handled symmetrically).
			This means that $\{i,j\} \in E(K)$ for all $i,j \in I_a$. By the definition of $K$, we have that $(V_i,V_j)$ is complete in $G$ for all $\{i,j\} \in \binom{I_a}{2} \setminus \mathcal{M}$. 
			By Item 2 in Lemma \ref{lem:cograph_partition_main} we have that $|V_i| \leq \lceil \alpha(n-1)\rceil$ for every $1 \leq i \leq t$. So $$\sum_{i \in I_a}{\binom{|V_i|}{2}} \leq \frac{n-1}{2} \cdot \lceil \alpha(n-1)\rceil \leq \alpha(n-1)^2$$ 
			and
			\begin{equation}\label{eq:matching_pairs sum}
			\sum_{\{i,j\} \in \mathcal{M}}{|V_i||V_j|} \leq \lceil \alpha(n-1) \rceil \cdot \sum_{\{i,j\} \in \mathcal{M}}{\min\{|V_i|,|V_j|\}} \leq \lceil \alpha(n-1) \rceil \cdot \frac{n-1}{2} \leq \alpha(n-1)^2.
			\end{equation}
			Moreover, the number of non-edges in $G$ which touch $V(G) \setminus W_a$ is at most $|V(G) \setminus W_a| \cdot (n-2) \leq \frac{\xi}{3}(n-1)^2$. 
			It follows that the total number of non-edges in $G$ is at most 
			$(\frac{\xi}{3} + 2\alpha)(n-1)^2$.
			Since $G$ and $L(v)$ differ on at most $\zeta (n-1)^2$ edges, the total number of non-edges in $L(v)$ is at most 
			$(\frac{\xi}{3} + 2\alpha + \zeta)(n-1)^2
			< \xi \binom{n-1}{2}$.
			Hence, $d(L(v)) > 1 - \xi$, contradicting condition 1.
		\end{proof}
		We will apply Lemma \ref{lem:weighted_cograph_partition} to the cograph $F$. Define the weight function $w$ on $[r] = V(F)$ by 
		$w(i) = \frac{|W_i|}{|W_1| + \dots + |W_r|}$. 
		Note that $|W_1| + \dots + |W_r| = n-1-|S| \geq (1 - 20\beta)(n-1)$, where the inequality holds by Item 1 of Lemma \ref{lem:cograph_partition_main}. 
		By Claim \ref{claim:non-trivial_partition}, we have 
		$w(i) \leq  
		\frac{1-\frac{\xi}{3}}{1 - 20\beta} \leq 1 - \frac{\xi}{4}$, 
		where in the last inequality we used our choice of $\beta$. 
		Apply Lemma \ref{lem:weighted_cograph_partition} to the cograph $F$ with parameter $\frac{\xi}{4}$ to obtain a partition $V(F) = [r] = \mathcal{I} \cup \mathcal{J} \cup \mathcal{L}$ as in that lemma. 
		Set $\mathcal{X} = \bigcup_{a \in \mathcal{I}}{W_a}$, 
		$\mathcal{Y} = \bigcup_{a \in \mathcal{J}}{W_a}$ and 
		$\mathcal{S} = \bigcup_{a \in \mathcal{L}}{W_a}$. Note that
		$|\mathcal{S}| \leq \xi n/4$ and 
		$|\mathcal{X}|,|\mathcal{Y}| \geq \frac{\xi}{8} \cdot (|W_1| + \dots + |W_r|) \geq \frac{\xi}{8}(1 - 20\beta)(n-1) \geq \frac{\xi}{8}(1 - 21\beta)n \geq \frac{\xi}{10}n$. Place the elements of $S \cup \{v\}$ into $\mathcal{S}$. After this step, we have $|\mathcal{S}| \leq \frac{\xi}{4}n + 20\beta n + 1 \leq \frac{\xi}{2} n$. 
		
		As guaranteed by Lemma \ref{lem:weighted_cograph_partition}, the bipartite graph between $\mathcal{I}$ and $\mathcal{J}$ (in $F$) is either complete or empty; suppose without loss of generality that it is complete (the other case can be handled symmetrically). 
		Set $I = \bigcup_{a \in \mathcal{I}}{I_a}$ and $J = \bigcup_{a \in \mathcal{J}}{I_a}$. In other words, $I$ is the set of $i \in [t]$ such that $V_i \subseteq \mathcal{X}$, and similarly $J$ is the set of $i \in [t]$ such that $V_i \subseteq \mathcal{Y}$. Note that the bipartite graph between $I$ and $J$ in the graph $K$ is complete. 
		\begin{claim}\label{claim:pairs}
			Let $i \in I, j \in J$ such that $\{i,j\} \notin \mathcal{M}$.
			Then 
			$d(V_i,V_i,V_j),d(V_j,V_j,V_i) \geq 1-\frac{\varepsilon^2}{32}$.
		\end{claim}
		\begin{proof}
			Since $i\not\sim j$, there is $k \in [t] \setminus \{i,j\}$ such that $\{i,k\} \in E(K), \{j,k\} \notin E(K)$ or $\{j,k\} \in E(K), \{i,k\} \notin E(K)$. Without loss of generality, assume that $\{i,k\} \in E(K), \{j,k\} \notin E(K)$. 
			

			Note that the bipartite graphs $(V_i,V_j),(U_k,V_i)$ are complete in $G$, and the bipartite graph $(U_k,V_j)$ is empty in $G$. 
			By \eqref{eq:V_i U_i}, we have $|V_i|,|V_j| \geq \alpha\beta (n-1)$ and $|U_k| \geq \alpha\beta^3(n-1)/2$. Since $G$ and $L(v)$ differ on at most $\zeta(n-1)^2$ edges, we have 
			$d_{L(v)}(V_i,U_k) \geq 1 - \frac{\zeta}{\alpha^2\beta^4/2} \geq 1 - \frac{\gamma^2}{80}$ and similarly 
			$d_{L(v)}(V_i,V_j) \geq 1 - \frac{\zeta}{\alpha^2\beta^2} \geq 1 - \frac{\gamma^2}{80}$ and $d_{L(v)}(V_j,U_k) \leq \frac{\zeta}{\alpha^2\beta^4/2} \leq \frac{\gamma^2}{80}$. Apply Lemma \ref{lem:homogeneous} with $X = V_i, Y = V_j, Z = U_k$. 
			Items 2 and 3 in Lemma \ref{lem:homogeneous} cannot hold by our choice of $\delta$, as $\min\{ |X|^2|Y||Z|, \; |X||Y|^2|Z|, \; |X|^2|Y|^2 \} \geq (\alpha\beta(n-1))^3 \cdot (\alpha\beta^3(n-1)/2) \geq 
			\alpha^4\beta^{6} n^4/4$. Hence, Item 1 must hold, giving $d(V_i,V_i,V_j),d(V_j,V_j,V_i) \geq 1-\gamma = 1 - \frac{\varepsilon^2}{32}$, as required. 
		\end{proof}
		\begin{claim}\label{claim:triples_YYX}
			Let $i \in I$ and $j,k \in J$ with $j \neq k$ and $\{i,j\},\{i,k\} \notin \mathcal{M}$. Then $d(V_i,V_j,V_k) \geq 1 - \frac{\varepsilon}{2}$.
		\end{claim}
		\begin{proof}
			By Claim \ref{claim:pairs}, we have $d(V_i,V_i,V_j),d(V_i,V_i,V_k) \geq 1-\frac{\varepsilon^2}{32}$. Now apply Lemma \ref{lem:pairs_to_triple} with $X = V_i, Y = V_j, Z = V_k$ and parameter $\frac{\varepsilon}{2}$. If Item 2 in Lemma \ref{lem:pairs_to_triple} holds then $H$ contains at least $\frac{(\varepsilon/2)^2}{16}|V_i|^2|V_j||V_k| \geq \frac{\varepsilon^2}{64}(\alpha\beta (n-1))^4 \geq \delta n^4$ induced copies of $D_2$, a contradiction. Hence Item 1 in Lemma \ref{lem:pairs_to_triple} holds, giving $d(V_i,V_j,V_k) \geq 1-\frac{\varepsilon}{2}$. 
		\end{proof}
		\noindent
		By symmetry, we have the following:
		\begin{claim}\label{claim:triples_XXY}
			Let $j,k \in I$, $i \in J$ with $j \neq k$ and $\{i,j\},\{i,k\} \notin \mathcal{M}$. Then $d(V_i,V_j,V_k) \geq 1 - \frac{\varepsilon}{2}$.
		\end{claim}
		By Claims \ref{claim:pairs} and \ref{claim:triples_YYX}, the number of triples $\{x,y_1,y_2\}$ with $x \in \mathcal{X}$ and $y_1,y_2 \in \mathcal{Y}$ which are non-edges of $H$, is at most 
		\begin{align*}\label{eq:e(Y,Y,X)}
		&\frac{\varepsilon}{2} \cdot \sum_{i \in I, j \in J}{|V_i|\binom{|V_j|}{2}} + 
		\frac{\varepsilon}{2} \cdot \sum_{i \in I, \; \{j,k\} \in \binom{J}{2}}
		{|V_i||V_j||V_k|} + 
		|\mathcal{Y}| \cdot
		\sum_{\substack{i \in I, \; j \in J: \\ \{i,j\} \in \mathcal{M}}}{|V_i||V_j|} = 
		\\ &\frac{\varepsilon}{2}|\mathcal{X}|\binom{|\mathcal{Y}|}{2} + 
		|\mathcal{Y}| \cdot 
		\sum_{\substack{i \in I, \; j \in J: \\ \{i,j\} \in \mathcal{M}}}{|V_i||V_j|} \leq 
		\frac{\varepsilon}{2}|\mathcal{X}|\binom{|\mathcal{Y}|}{2} + \alpha n^2 |\mathcal{Y}|,
		\end{align*}
		where in the last inequality we used \eqref{eq:matching_pairs sum}. It follows that 
		$$
		d(\mathcal{Y},\mathcal{Y},\mathcal{X}) \geq 1 - \frac{\varepsilon}{2} - \frac{\alpha n^2|\mathcal{Y}|}{|\mathcal{X}|\binom{|\mathcal{Y}|}{2}} \geq 
		1 - \frac{\varepsilon}{2} - \frac{4\alpha n^2}{|\mathcal{X}||\mathcal{Y}|} \geq 
		1 - \frac{\varepsilon}{2} - \frac{4\alpha n^2}{(\xi n/10)^2} \geq 1 - \varepsilon.
		$$
		Here, the last inequality uses our choice of $\alpha$. 
		Similarly, by using Claims \ref{claim:pairs} and \ref{claim:triples_XXY}, one establishes that $d(\mathcal{X},\mathcal{X},\mathcal{Y}) \geq 1 - \varepsilon$. Hence, $(\mathcal{X},\mathcal{Y})$ is $\varepsilon$-homogeneous, as required. 
	\end{proof} 
	
	\section{Proof of Theorem \ref{thm:removal}}\label{sect:removal}
	\noindent
	We start with the following corollary of Lemma \ref{lem:main}. 
	\begin{lemma}\label{lem:main_testing}
		For every $\varepsilon > 0$ there are $\delta = \delta(\varepsilon) = \poly(\varepsilon) > 0$ and $n_0 = n_0(\varepsilon) = \poly(1/\varepsilon)$ such that the following holds. Let $H$ be a $3$-uniform hypergraph on $n \geq n_0$ vertices which contains less than $\delta n^4$ induced copies of $D_2$. Then there is a partition $V(H) = \mathcal{X} \cup \mathcal{Y}$, $\mathcal{X},\mathcal{Y} \neq \emptyset$, such that $(\mathcal{X},\mathcal{Y})$ is $\varepsilon$-homogeneous. 
	\end{lemma}
	\begin{proof}
		If there is a vertex $v \in V(H)$ with $\deg(v) \geq (1-\varepsilon)\binom{n-1}{2}$ or $\deg(v) \leq \varepsilon \binom{n-1}{2}$, then the partition $\{v\},V(H) \setminus \{v\}$ satisfies the assertion of the lemma. Otherwise, for every $v \in V(H)$ we have $\varepsilon \binom{n-1}{2} \leq \deg(v) \leq (1 - \nolinebreak \varepsilon) \binom{n-1}{2}$, or, equivalently, $\varepsilon \leq d(L(v)) \leq 1 - \varepsilon$. By averaging, there is a vertex $v \in V(H)$ which participates in less than $\delta n^3$ induced copies of $D_2$. Let us fix such a vertex $v$. Now all conditions of Lemma \ref{lem:main} are satisfied. Apply Lemma \ref{lem:main} with parameters $\frac{\varepsilon}{2}$ and $\xi = \frac{\varepsilon}{16}$ to obtain a partition 
		$V(H) = \mathcal{X}' \cup \mathcal{Y}' \cup \mathcal{S}$ with 
		$|\mathcal{S}| \leq \frac{\varepsilon}{32}n$ and $(\mathcal{X}',\mathcal{Y}')$ is $\frac{\varepsilon}{2}$-homogeneous. Without loss of generality, suppose that $|\mathcal{X}'| \leq |\mathcal{Y}'|$; so $|\mathcal{Y}'| \geq (1 - \varepsilon)n/2$. Set $\mathcal{X} = \mathcal{X}'$ and $\mathcal{Y} = \mathcal{Y}' \cup \mathcal{S}$. 
		
		Let us assume, without loss of generality, that 
		$d(\mathcal{X}',\mathcal{X}',\mathcal{Y}'), d(\mathcal{Y}',\mathcal{Y}',\mathcal{X}') \geq 1 - \frac{\varepsilon}{2}$. We have
		$$
		d(\mathcal{X},\mathcal{X},\mathcal{Y}) \geq 1 - \frac{\varepsilon}{2} - 
		\frac{|\mathcal{S}|}{|\mathcal{Y}'|} 
		\geq 
		1 - \frac{\varepsilon}{2} - \frac{\varepsilon/32}{(1 - \varepsilon)/2} \geq 1 - \varepsilon,
		$$
		and
		$$
		d(\mathcal{Y},\mathcal{Y},\mathcal{X}) \geq 1 - \frac{\varepsilon}{2} - 
		\frac{|\mathcal{S}||\mathcal{Y}'|}{\binom{|\mathcal{Y}'|}{2}}
		\geq 
		1 - \frac{\varepsilon}{2} - \frac{4|\mathcal{S}|}{|\mathcal{Y}'|} \geq 
		1 - \frac{\varepsilon}{2} - \frac{\varepsilon/8}{(1 - \varepsilon)/2} \geq 1 - \varepsilon.
		$$
	\end{proof}
	
	Define the family of \emph{cohypergraphs} as follows. The 3-uniform hypergraph $H$ is a cohypergraph if $|V(H)|=1$, or $V(H)$ has a partition $X\cup Y$ with $X,Y\neq \emptyset$ such that $(X,Y)$ is homogeneous and $H[X], H[Y]$ are cohypergraphs. The next lemma easily follows from the definitions.
	
	\begin{lemma}\label{lem:partition}
		If $H$ is a cohypergraph, then $H$ is $D_2$-free.
	\end{lemma} 
	
	\begin{proof}
	    We prove this by induction on $|V(H)|$. The base case $|V(H)|=1$ is trivial, so assume $|V(H)|\geq 2$. Then there exists a partition $X\cup Y$ with $X,Y\neq \emptyset$ such that $(X,Y)$ is homogeneous and $H[X], H[Y]$ are cohypergraphs, and therefore $D_2$-free. But then $H$ is also $D_2$-free.
	\end{proof}
	
	We remark that the converse is not true, that is, not every $D_2$-free hypergraph is a cohypergraph. E.g. linear 3-graphs are $D_2$-free but not necessarily cohypergraphs. However, we prove that if a hypergraph contains few induced copies of $D_2$, it can be made a cohypergraph by changing \nolinebreak few \nolinebreak edges.
	
	\begin{proof}[Proof of Theorem \ref{thm:removal}]
		Let $\delta'$ denote the $\delta(\varepsilon)$ given by Lemma \ref{lem:main_testing}. We show that in Theorem \ref{thm:removal}, one can take $\delta = \delta'\varepsilon^4$. 
		We decompose $H$ by repeatedly applying Lemma \ref{lem:main_testing}, as follows. It is convenient to describe the decomposition using a tree, where each node corresponds to a subset of $V(H)$. The root is $V(H)$. At each step, if there is a leaf $X$ with $|X| \geq \varepsilon n$, then apply Lemma \ref{lem:main_testing} to $H[X]$, noting that $H[X]$ contains less than $\delta' |X|^4$ induced copies of $D_2$ by our choice of $\delta$. 
		Lemma \ref{lem:main_testing} gives a partition $X = X_1 \cup X_2$ such that $(X_1,X_2)$ is $\varepsilon$-homogeneous. Now add $X_1,X_2$ as the children of $X$. When this process stops, every leaf is of size less than $\varepsilon n$. For each non-leaf $X$, make the pair $(X_1,X_2)$ homogeneous by adding/deleting at most $\varepsilon \cdot \big( \binom{|X_1|}{2}|X_2| + \binom{|X_2|}{2}|X_1| \big)$ edges. This requires at most $\varepsilon\binom{n}{3}$ edge changes in total. Next, for each leaf $X$, delete every edge of $H[X]$. This requires at most 
		$$
		\sum_{X \text{ leaf}}{\binom{|X|}{3}} < \frac{\varepsilon^2 n^2}{6} \cdot
		\sum_{X \text{ leaf}}{|X|} = \frac{\varepsilon^2 n^3}{6}
		$$ 
		additional edge changes. So the total number of edge-changes is at most $\varepsilon\binom{n}{3} + \frac{\varepsilon^2 n^3}{6} \leq \varepsilon n^3$. After these edge-changes, it is easy to see that the resulting hypergraph is a cohypergraph, so it is also induced $D_2$-free by Lemma \ref{lem:partition}. This completes the proof. 
	\end{proof}

\section{Proof of Theorem \ref{thm:EH}}\label{sect:EH}
	We start with some preliminary lemmas. The following lemma shows that if the density of a hypergraph is bounded away from $0$ and $1$, then there is a vertex whose link graph also has density bounded away from $0$ and $1$. 
	\begin{lemma}\label{lem:link_graph_density}
	Let $\beta > 0$, and let $H$ be an $n$-vertex $3$-uniform hypergraph 
	with $n \geq 4/\beta$ and $\beta \leq \nolinebreak d(H) \leq 1 - \beta$. Then there is $v \in V(H)$ with $\frac{\beta^2}{16} \leq d(L(v)) \leq 1 - \frac{\beta^2}{16}$.
	\end{lemma}
\begin{proof}
	Let 
	$X = \{ v \in V(H) : d(L(v)) < \frac{\beta^2}{16} \}$, 
	$Y = \{ v \in V(H) : d(L(v)) > 1 - \frac{\beta^2}{16} \}$
	and $Z = V(H) \setminus (X \cup Y)$.
	First we show that $|X| < \frac{\beta}{2} n$ or $|Y| < \frac{\beta}{2} n$. So suppose that $|X| \geq \frac{\beta}{2} n$, and let us show that $|Y| < \frac{\beta}{2} n$. Observe that 
	$$
	e(X,X,Y) < e(X,V(H),V(H)) < |X| \cdot \frac{\beta^2}{16} \cdot \binom{n-1}{2},
	$$ 
	and on the other hand
	$$
	e(X,X,Y) > |Y| \cdot \left( \binom{|X|}{2} - \frac{\beta^2}{16} \cdot\binom{n-1}{2} \right)
	\geq 
	|Y| \cdot \left( \frac{|X|^2}{4} - \frac{\beta^2}{16} \cdot \frac{n^2}{2} \right) \geq |Y| \cdot \frac{|X|^2}{8}.
	$$	
	Combining these two inequalities, we get
	$|Y| \cdot \frac{|X|^2}{8} < |X| \cdot \frac{\beta^2}{16} \cdot \frac{n^2}{2}$ and hence $|Y| < \frac{\beta^2}{4} n^2/|X| \leq \frac{\beta}{2}n$. So indeed $|X| < \frac{\beta}{2} n$ or $|Y| < \frac{\beta}{2} n$. Let us assume that $|Y| < \frac{\beta}{2}n$; the other case is symmetric.
	Now, if $Z = \emptyset$ then 
	$e(H) = \frac{1}{3}\sum_{v \in V(H)}{\deg(v)} \leq \frac{1}{3} \cdot \left( |Y| \cdot \binom{n-1}{2} + |X| \cdot \frac{\beta^2}{16} \cdot \binom{n-1}{2} \right) < \frac{1}{3}\beta n \cdot \binom{n-1}{2} = \beta \binom{n}{3}$. But this implies that $d(H) < \beta$, a contradiction. Hence, $Z \neq \emptyset$, as required. 
\end{proof}

From now on, let $C_0>1$ be an absolute constant such that 
$n_{\ref{lem:main}}(\xi,\varepsilon) \leq C \cdot (\xi\varepsilon)^{-C}$ for every $C>C_0$, where $n_{\ref{lem:main}}$ is from Lemma \ref{lem:main}. 
The next lemma follows easily from Lemma \ref{lem:main}. 

\begin{lemma}\label{lem:main_EH}
	Let $\beta > 0$, $C>C_0$, and let $H$ be an induced $D_2$-free $3$-uniform hypergraph on $n \geq 4/\beta$ vertices satisfying $\beta \leq d(H) \leq 1 - \beta$. Then there is a partition $V(H) = \mathcal{X} \cup \mathcal{Y} \cup \mathcal{S}$ such that
	$|\mathcal{X}|,|\mathcal{Y}| \geq \xi n/10$, $|\mathcal{S}| \leq \xi n/2$ and $(\mathcal{X},\mathcal{Y})$ is $\varepsilon$-homogeneous, where $\xi = \frac{\beta^2}{16}$ and
	$\varepsilon = \frac{1}{\xi} \cdot \left( C/n \right)^{1/C}$.
\end{lemma}
\begin{proof}
	By the choice of $\varepsilon$, we have $n \geq C \cdot (\xi\varepsilon)^{-C} \geq n_{\ref{lem:main}}(\xi,\varepsilon)$. By Lemma \ref{lem:link_graph_density}, there is $v \in V(H)$ with $\xi \leq d(L(v)) \leq 1 - \xi$. Applying Lemma \ref{lem:main} completes the proof. 
\end{proof}

\noindent
We will need the following well-known simple probabilistic lemma, see \cite[Exercise 3 in Chapter 3]{AlonSpencer}. For completeness, we include a proof. 

\begin{lemma}\label{lem:probabilistic_independent_set}
	Let $H$ be a $3$-uniform hypergraph with $n$ vertices and density $0 \leq d \leq 1$. Then $H$ contains an independent set of size at least 
	$\min\left\{ \frac{n}{2}, \sqrt{\frac{3}{4d}} \right\}$. By symmetry, $H$ contains a clique of size at least $\min\left\{ \frac{n}{2}, \sqrt{\frac{3}{4(1-d)}} \right\}$.
\end{lemma}

\begin{proof}
	If $d < \frac{3}{n^{2}}$, then $H$ contains at most $\frac{n}{2}$ edges, so by removing a vertex from each edge we get an independent set of size at least $\frac{n}{2}$.  Suppose that $d \geq \frac{3}{n^{2}}$, then $p = \sqrt{\frac{3}{d}} \cdot \frac{1}{n}$ satisfies $p\leq 1$. Sample a subset $X \subseteq V(H)$ by choosing each vertex with probability $p$ independently. Delete one vertex from each edge contained in $X$ to obtain an independent set. The expected size of $X$ is $np$, and the expected number of edges contained in $X$ is $p^3 d\binom{n}{3} \leq p^3dn^3/6$. Hence, there is a choice of $X$ for which the resulting independent set has size at least $np - p^3dn^3/6 = np/2 = \sqrt{\frac{3}{4d}}$.  
\end{proof}

Let us give an outline of the proof of Theorem \ref{thm:EH}. The main idea is to apply Lemma \ref{lem:main_EH} with $\varepsilon = n^{-c}$, for some sufficiently small constant $c > 0$. This way we get a partition of the vertex-set into $\mathcal{X},\mathcal{Y}$ and a ``leftover set'' $\mathcal{S}$ such that $(\mathcal{X},\mathcal{Y})$ is $\varepsilon$-homogeneous and $\mathcal{S}$ is small. We continue decomposing the hypergraph by applying Lemma \ref{lem:main_EH} to the hypergraphs induced by $\mathcal{X},\mathcal{Y}$ and so on, until all parts are sufficiently small. By choosing the parameters appropriately, we can make sure that the union of all leftover sets $\mathcal{S}$ accumulated throughout the process is small. We then consider an auxiliary graph on the set of parts obtained in the process, in which two parts are adjacent if their density is close to $1$ and non-adjacent if their density is close to $0$. This graph is a cograph. Using the fact that every cograph $G$ has a clique or independent set of size $\sqrt{|V(G)|}$, we obtain a set $\mathcal{A}$ of parts such that either all pairs $X,Y \in \binom{\mathcal{A}}{2}$ have density close to $1$, or all such pairs have density close to $0$. We then use Lemma \ref{lem:pairs_to_triple} to deduce that all triples $X,Y,Z \in \binom{\mathcal{A}}{3}$ have density close to $1$ or all have density close to $0$. This way we obtain a set, namely $\bigcup_{X \in \mathcal{A}} X$, which has density at least $1 - n^{-c}$ or at most $n^{-c}$, where $c > 0$ is an appropriate small constant. Finally, using Lemma \ref{lem:probabilistic_independent_set}, we get the required polynomial-size clique or independent set. 

We would like to draw the reader's attention to the following subtlety in the outline above. Suppose that at some step of the process, we found an $\varepsilon$-homogeneous pair $(\mathcal{X}_1,\mathcal{Y}_1)$, and then went on to decompose $\mathcal{X}_1$ and $\mathcal{Y}_1$ further. Suppose that $\mathcal{X}_2 \subseteq \mathcal{X}_1$ and $\mathcal{Y}_2 \subseteq \mathcal{Y}_1$ are parts found later in the process. The point is that we want the pair $(\mathcal{X}_2,\mathcal{Y}_2)$ to also be nearly-homogeneous (i.e. $\varepsilon'$-homogeneous for some small $\varepsilon'$). To deduce this from the $\varepsilon$-homogeneity of $(\mathcal{X}_1,\mathcal{Y}_1)$, we need $\varepsilon$ to be much smaller than 
$|\mathcal{X}_2|/|\mathcal{X}_1|$ and $|\mathcal{Y}_2|/|\mathcal{Y}_1|$. This will indeed be true because the sets $\mathcal{X},\mathcal{Y}$ given by Lemma \ref{lem:main} are guaranteed to occupy at least a $\frac{\xi}{10}$-fraction of the entire vertex-set, and $\varepsilon$ will be chosen much smaller than $\xi$. This is the reason that in Lemma \ref{lem:main} we need two parameters, one governing the sizes of $\mathcal{X}$ and $\mathcal{Y}$, and one the homogeneity. 
\begin{proof}[Proof of Theorem \ref{thm:EH}]
	We may and will assume that $n$ is large enough where needed. Let $\eta=n^{-1/(100C_0)}$, and define the following parameters: 
	$$
	\gamma:= \frac{\eta^2}{8}, 
	\; \; \;
	t:= \frac{2}{\gamma^3} = \frac{2^{10}}{\eta^6}, 
	\; \; \;
	\beta: = 2\gamma^{3/2} = \frac{\eta^{3}}{2^{7/2}},
	\; \; \;
	\xi: = \frac{\beta^2}{16} = \frac{\gamma^3}{4} = \frac{1}{2t}=\frac{\eta^{6}}{2^{11}},
	$$
	and
	\begin{equation}\label{eq:EH_parameter_choice}
	\varepsilon: = \frac{(\xi/10)^3\gamma^4}{2} = \frac{\eta^{26}}{10^{3}\cdot 2^{46}}.
	\end{equation}
    Then we can choose $C$ such that $C_0<C<10C_0$ and 
	$$
	\varepsilon = \frac{1}{\xi} \cdot \left( \frac{C}{\gamma n} \right)^{1/C}=2^{11} (8C)^{1/C}\eta^{-6-2/C} n^{-1/C}.
	$$
	also holds. Note that $\xi\gamma n > n^{1/2}$. 
	
 Let $H$ be a $D_2$-free 3-uniform hypergraph on $n$ vertices.	We will prove that $H$ contains a homogeneous set of size $\Omega(\sqrt{1/\eta})=\Omega(n^{1/(200C_0)})$. We may assume that every subset $U \subseteq V(H)$ of size $|U| \geq \gamma n$ satisfies $\beta \leq d(H[U]) \leq 1 - \beta$; otherwise we are done by Lemma~\ref{lem:probabilistic_independent_set}. 
	
	We decompose $H$ by repeatedly applying Lemma \ref{lem:main_EH}, as follows. 
	At each step, we will have a partition $\mathcal{P}_i$ of a subset of $V(H)$. Set $\mathcal{P}_0 = \{V(H)\}$. 
	We also maintain a cograph $\mathcal{G}$, whose vertex set will be $\mathcal{P}_i$ (initially $\mathcal{G}$ is the 1-vertex graph). For $i = 1,\dots,t$, do as follows: if for every $U \in \mathcal{P}_{i-1}$ it holds that $|U| \leq \gamma n$, then stop. 
	Otherwise, let $U_i \in \mathcal{P}_{i-1}$ with $|U_i| > \gamma n$. 
	Apply Lemma \ref{lem:main_EH} to $H[U_i]$ to obtain a partition $U_i = \mathcal{X}_i \cup \mathcal{Y}_i \cup \mathcal{S}_i$ such that $|\mathcal{X}_i|,|\mathcal{Y}_i| \geq \xi|U_i|/10 \geq \xi \gamma n/10$, $|\mathcal{S}_i| \leq \xi|U_i| \leq \xi n$, and $(\mathcal{X}_i,\mathcal{Y}_i)$ is $\varepsilon$-homogeneous. 
	Define $\mathcal{P}_i = (\mathcal{P}_{i-1} \setminus \{U_i\}) \cup \{\mathcal{X}_i,\mathcal{Y}_i\}$. In $\mathcal{G}$, we replace the vertex $U_i$ with $\mathcal{X}_i,\mathcal{Y}_i$, connecting each 
	$W \in \mathcal{P}_{i-1} \setminus \{U_i\}$ to both $\mathcal{X}_i,\mathcal{Y}_i$ if $W$ was connected to $U_i$, and to none of $\mathcal{X}_i,\mathcal{Y}_i$ if $W$ was not connected to $U_i$. 
	We let $\{\mathcal{X}_i,\mathcal{Y}_i\}$ be an edge of $\mathcal{G}$ if 
	$d(\mathcal{X}_i,\mathcal{X}_i,\mathcal{Y}_i), d(\mathcal{Y}_i,\mathcal{Y}_i,\mathcal{X}_i) \geq 1 - \varepsilon$, and a non-edge if 
	$d(\mathcal{X}_i,\mathcal{X}_i,\mathcal{Y}_i), d(\mathcal{Y}_i,\mathcal{Y}_i,\mathcal{X}_i) \leq \varepsilon$. 
	It is easy to see that $\mathcal{G}$ is a cograph throughout the process. 
	
	\noindent
	We say that step $i$ is {\em good} if $|\mathcal{X}_i| \leq \gamma |U_i|$ or $|\mathcal{Y}_i| \leq \gamma |U_i|$. Otherwise, step $i$ is {\em bad}.
 	\begin{claim}\label{claim:iterated_partition}
 		The number of bad steps is at most $1/\gamma^3$. 
 	\end{claim}
 	\begin{proof}
 		For $i \geq 0$, define $q_i := \sum_{U \in \mathcal{P}_i}{(|U|/n)^2}$. Evidently, $0 \leq q_i \leq 1$ for every $i$. Note that for $i \geq 1$, we have
 		$q_{i-1} - q_i = (|U_i|/n)^2 - (|\mathcal{X}_i|/n)^2 - (|\mathcal{Y}_i|/n)^2 > 0$, where $U_i,\mathcal{X}_i,\mathcal{Y}_i$ are as above. Also, if step $i$ is bad then
 		$$
 		q_{i-1} - q_i \geq (|U_i|/n)^2 \cdot \left( 1 - \gamma^2 - (1 - \gamma)^2 \right) 
 		= 
 		(|U_i|/n)^2 \cdot 2(\gamma - \gamma^2) \geq (|U_i|/n)^2 \cdot \gamma \geq \gamma^3.
 		$$
		So we see that if the number of bad steps were larger than $1/\gamma^3$, then we would have $q_i < 0$ for some $i$, a contradiction.
 	\end{proof}
 	Suppose that the process ran for $s$ steps, where $1 \leq s \leq t$. 
 	Let $\mathcal{S}$ be the union of the sets $\mathcal{S}_i$, $1 \leq i \leq s$. Note that $|\mathcal{S}| \leq t \cdot \xi n=\frac{n}{2}$. 
 	We now show that at the end of the process, there are many parts of size at most $\gamma n$. 
 	\begin{claim}\label{claim:good_parts}
 		There are at least $\frac{1}{2\gamma}$ sets $U \in \mathcal{P}_s$ satisfying $|U| \leq \gamma n$. 
 	\end{claim}
 	\begin{proof}
 		Suppose first that $s < t$, namely that the process stopped before step $t$. Then by definition, we must have $|U| \leq \gamma n$ for every $U \in \mathcal{P}_s$. Also, $|\bigcup_{U \in \mathcal{P}_s}{U}| = n - |\mathcal{S}| \geq \frac{n}{2}$. It follows that $|\mathcal{P}_s| \geq \frac{1}{2\gamma}$, completing the proof in this case.
 		
 		Suppose now that $s = t$. By Claim \ref{claim:iterated_partition}, the number of bad steps is at most $1/\gamma^3$. Hence, the number of good steps is at least $t - 1/\gamma^3 \geq \frac{1}{2\gamma}$. Observe that if step $i$ is good, then after this step, we have 
 		$\#\{U \in \mathcal{P}_i : |U| \leq \gamma n\} \geq 
 		\#\{U \in \mathcal{P}_{i-1} : |U| \leq \gamma n\} + 1$. 
 		It follows that there are at least $\frac{1}{2\gamma}$ sets $U \in \mathcal{P}_t$ satisfying $|U| \leq \gamma n$, as required.  
 	\end{proof}
 	Let $\mathcal{U}$ be the set of all $U \in \mathcal{P}_s$ with $|U| \leq \gamma n$. By Claim \ref{claim:good_parts}, we have $|\mathcal{U}| \geq \frac{1}{2\gamma}$. Recall that $|U| \geq \xi\gamma n/10$ for every $U \in \mathcal{U}$. 
 	Consider the intervals $I_i = [2^{i}\xi\gamma n/10, 2^{i+1}\xi\gamma n/10]$ for $0 \leq i \leq \log(10/\xi) - 1$. Note that $|U|$ belongs to one of these intervals for every $U \in \mathcal{U}$. By the pigeonhole principle, there are $\mathcal{U}' \subseteq \mathcal{U}$ and $0 \leq i \leq \log(10/\xi) - 1$ such that $|\mathcal{U}'| \geq |\mathcal{U}|/\log(10/\xi) \geq \frac{1}{2\gamma \log(10/\xi)} \geq 1/\gamma^{1/2}$, and $2^{i}\xi\gamma n/10 \leq |U| \leq 2^{i+1}\xi\gamma n/10$ for every $U \in \mathcal{U}'$. 
 	Consider the subgraph of $\mathcal{G}$ induced by $\mathcal{U}'$. As $\mathcal{G}$ is a cograph, there is $\mathcal{A} \subseteq \mathcal{U}'$ with $|\mathcal{A}| \geq |\mathcal{U}'|^{1/2} \geq 1/\gamma^{1/4}$ such that $\mathcal{A}$ is a clique or an independent set in $\mathcal{G}$. Let us assume that $\mathcal{A}$ is a clique; the case that $\mathcal{A}$ is an independent set is symmetric. 
 	Set $V = \bigcup_{U \in \mathcal{A}}{U}$. Note that 
 	\begin{equation}\label{eq:EH_|V|}
 	|V| \geq |\mathcal{A}| \cdot 2^i \xi \gamma n/10 \geq 1/\gamma^{1/4} \cdot 2^i \xi \gamma n/10 = 2^i \xi \gamma^{3/4} n/10.
 	\end{equation}
 	\begin{claim}\label{claim:EH_almost_homogeneous}
 		$d(H[V]) \geq 1 - 3\eta$.
 	\end{claim}
 	\begin{proof}
 		Fix any distinct $X,Y \in \mathcal{A}$. By construction, there is a step $1 \leq i \leq s$ such that $X \subseteq \mathcal{X}_i$ and $Y \subseteq \mathcal{Y}_i$. Also, since $\{X,Y\}$ is an edge of $\mathcal{G}$, it must be that $\{\mathcal{X}_i,\mathcal{Y}_i\}$ is an edge of $\mathcal{G}$ at step $i$. By definition, this means that 
 		$d(\mathcal{X}_i,\mathcal{X}_i,\mathcal{Y}_i), d(\mathcal{Y}_i,\mathcal{Y}_i,\mathcal{X}_i) \geq 1 - \varepsilon$. It follows that
 		$$
 		d(X,X,Y) \geq 1 - 
 		\frac{\varepsilon \binom{|\mathcal{X}_i|}{2}|\mathcal{Y}_i|}{\binom{|X|}{2}|Y|} \geq
 		1 -  
 		\frac{2\varepsilon |\mathcal{X}_i|^2|\mathcal{Y}_i|}{|X|^2|Y|} \geq
 		1 - 
 		\frac{2\varepsilon n^3}{(\xi\gamma n/10)^3} = 1 - \gamma,
 		$$
 		where the last equality uses \eqref{eq:EH_parameter_choice}. 
 		Similarly, $d(Y,Y,X) \geq 1 - \gamma$. 
 		
 		Now fix any distinct $X,Y,Z \in \mathcal{A}$. As $d(X,X,Y),d(X,X,Z) \geq 1 - \gamma = 1 - \frac{\eta^2}{8}$, we get from Lemma \ref{lem:pairs_to_triple} that $d(X,Y,Z) \geq 1 - \eta$. It now follows that 
		$$
		d(H[V]) \geq 1 - \eta
		- \sum_{X \in \mathcal{A}}{\frac{\binom{|X|}{3}}{\binom{|V|}{3}}} \; .
		$$
		Recall that for every $X \in \mathcal{A}$, we have $|X| \leq 2^{i+1}\xi\gamma n/10$. Combining this with \eqref{eq:EH_|V|}, we get 
		$$
		\sum_{X \in \mathcal{A}}{\frac{\binom{|X|}{3}}{\binom{|V|}{3}}} \leq 
		\sum_{X \in \mathcal{A}}{\left( \frac{|X|}{|V|} \right)^3} \leq 
		(2\gamma^{1/4})^{2} \cdot \sum_{X \in \mathcal{A}}{\frac{|X|}{|V|}} = 4\gamma^{1/2} \leq 2\eta.
		$$
		So we have $d(H[V]) \geq 1 - 3\eta$, as required. 
  	\end{proof}
  	By Claim \ref{claim:EH_almost_homogeneous} and Lemma \ref{lem:probabilistic_independent_set}, $H[V]$ contains a clique of size 
  	$\min\{ \frac{|V|}{2}, \, \Omega(\sqrt{1/\eta})\}$. As $|V| \geq \xi\gamma^{3/4}n/10 \geq n^{1/2}$, this gives a clique of size $\Omega(\sqrt{1/\eta})$, completing the proof. 
\end{proof}

    \section{Concluding remarks}
    In this paper we considered $3$-uniform hypergraphs which have no $4$ vertices spanning exactly $2$ edges. More generally, for integers $3 \leq k < \ell$ and a set $S \subseteq \{1,\dots,\binom{\ell}{k}\}$, one can study $k$-uniform hypergraphs which have no $\ell$ vertices which span exactly $s$ edges for some $s \in S$. What can be said about the Erd\H{o}s-Hajnal properties of such hypergraphs? Do polynomial Erd\H{o}s-Hajnal bounds hold in the case $\ell = k+1, S = \{2,3,\dots,k-1\}$ for every $k$?

	\paragraph{Acknowledgements} All authors were supported by the SNSF grant 200021\_196965. Istv\'an Tomon also acknowledges the support of Russian Government in the framework of MegaGrant no 075-15-2019-1926, and the support of MIPT Moscow.

\end{document}